\documentclass[11pt]{article}
\usepackage[margin=1in]{geometry}

\usepackage[pdftex]{graphicx}
\graphicspath{{../pdf/}{../jpeg/}}
\DeclareGraphicsExtensions{.pdf,.jpeg,.png}
\usepackage{amssymb,amsmath,url}
\usepackage{bm} % bold symbols
\usepackage{multirow}
\usepackage{booktabs}	% nice tables
\usepackage{color}
\usepackage{subfigure}
\usepackage{bigfoot}		% for footnote to look nice
\usepackage{algpseudocode}
\usepackage{pifont}
\usepackage{tikz}
\usepackage{color}
\usepackage{cite}
\usepackage{setspace}
\usepackage{lipsum}
\usepackage{algorithm} 
\usepackage{algpseudocode}
\usepackage[pdftex]{graphicx}
\usepackage{epsfig}
\usepackage{comment}
\usepackage{enumitem} 
\usepackage{amssymb}
\usepackage{hyperref}

\algrenewcommand\algorithmicrequire{\textbf{Input:}}
\algrenewcommand\algorithmicensure{\textbf{Output:}}
\algrenewcommand\algorithmicforall{\textbf{For}}

\newtheorem{theorem}{Theorem}
\newtheorem{lemma}{Lemma}
\newtheorem{proposition}{Proposition}
\newtheorem{corollary}{Corollary}

\allowdisplaybreaks
\newtheorem{definition}{Definition} 
\newtheorem{proof}{Proof}

\newtheorem{remark}{Remark}

\newtheorem{fact}{Fact}

\newtheorem{assumption}{Assumption}
\newtheorem{example}{Example}

\newcommand{\R}{\mathbb{R}}

\newcommand*{\QEDA}{\hfill\ensuremath{\blacksquare}}

\makeatletter
\newcommand\blfootnote[1]{%
  \begingroup
  \renewcommand\thefootnote{}\footnote{#1}%
  \addtocounter{footnote}{-1}%
  \endgroup
}
\makeatother
\title{\LARGE \bf 
	Actuator Placement for Structural Controllability\\ beyond Strong Connectivity and towards Robustness
}
\author{Baiwei Guo, Orcun Karaca, Sepide Azhdari, Maryam Kamgarpour, Giancarlo Ferrari-Trecate%<-this
\blfootnote{The work of B. Guo and G. Ferrari-Trecate received support from the Swiss National Science Foundation under the NCCR Automation (grant agreement 51NF40\_180545).\newline\indent B. Guo, S. Azhdari and G. Ferrari-Trecate are with Automatic Control Laboratory, Institute of Mechanical Engineering, École Polytechnique Fédérale de Lausanne, Switzerland, e-mails: {\{\tt baiwei.guo, sepide.azhdari, giancarlo.ferraritrecate\}@epfl.ch}\newline\indent O. Karaca is with ABB Corporate Research Center, Baden, Switzerland. email:  {{\tt orcun.karaca@ch.abb.com}}\newline\indent
M. Kamgarpour is with the Electrical and Computer Eng. at the University of British Columbia, Canada. email: {{\tt maryamk@ece.ubc.ca}}}}
\allowdisplaybreaks
\begin{document}
\maketitle 
%\thispagestyle{empty}
%\pagestyle{empty}
%%%%%%%%%%%%%%%%%%%%%%%%%%%%%%%%%%%%%%%%%%%%%%%%%%%%%%%%%%%%%%%%%%%%%%%%%%%%%%%%
\begin{abstract}\noindent Actuator placement is a fundamental problem in control design for large-scale networks. In this paper, we study the problem of finding a set of actuator positions by minimizing a given metric, while satisfying a structural controllability requirement and a constraint on the number of actuators. We first extend the classical forward greedy algorithm for applications to graphs that are not necessarily strongly connected. We then improve this greedy algorithm by extending its horizon. This is done by evaluating the actuator position set expansions at the further steps of the classical greedy algorithm. We prove that this new method attains a better performance, when this evaluation considers the final actuator position set. {Moreover, we study the problem of minimal backup placements. The goal is to ensure that the system stays structurally controllable even when any of the selected actuators goes offline, with minimum number of backup actuators. We show that this problem is equivalent to the well-studied NP-hard hitting set problem.} Our results are verified by a numerical case study.
\end{abstract}
%%%%%%%%%%%%%%%%%%%%%%%%%%%%%%%%%%%%%%%%%%%%%%%%%%%%%%%%%%%%%%%%%%%%%%%%%%%%%%%%
\vspace{.1cm}
\section{Introduction}
\vspace{.1cm}
The steady progress in computation and communication technologies is enabling the deployment of large networks of systems, which require optimal and robust coordination. Prominent examples include power grids \cite{summers2015submodularity} and industrial control systems \cite{banerjee1995control}. A significant amount of research is currently focusing on the control design for such networks in order to achieve better performance and security~\cite{magnanti1984network, chabarek2008power}. A fundamental design problem concerns actuator placement which aims to select a subset from all possible actuator positions to place actuators such that a chosen network metric is minimized.

Typical metrics are set functions that map the actuator positions to the costs for completing certain tasks. In \cite{summers2015submodularity}, the authors study controllability metrics to evaluate the energy required to bring the system to an arbitrary state. Whereas, variants of the LQG cost are considered in~\cite{tzoumas2020lqg}. Since some of these metrics are neither submodular nor supermodular~\cite{SummersPerformance2019}, it is in general NP-hard to find the optimal set of actuators positions~\cite{EffortBounds}.\footnote{Many works concerning such nonmodular objectives instead consider approximate submodularity or supermodularity properties defined by certain ratios, see for instance~\cite{bian2017guarantees,Das:2011:SMS:3104482.3104615,karaca2018exploiting,karaca2020ej}.} 

The sets minimizing these aforementioned metrics are not guaranteed to make the resulting system controllable, considering that either the metrics are not related to controllability objective at all, or even when they are related, they involve some approximation techniques~\cite{guo2020actuator,guo2019actuator}. In view of this issue, structurally controllable systems constitute a desirable class. These systems are the ones that attain controllability after a slight perturbation of the system parameters corresponding to edge weights in the underlying network graph \cite{SturcturalcontrollabilityTai,ramos2020structural}. The concept of structural controllability, along with related variants \cite{li2020structural,jia2020unifying}, is based only on the graphical interconnection structure of the dynamical system and the actuator positions. The works ~\cite{liu2011controllability} and~\cite{commault2013input} study the minimal number of input nodes under the constraint of structural controllability. Another related research direction studies this problem with the same constraint but beyond modular metrics. The work in \cite{clark2012leader} considers leader selection to minimize control errors due to noisy communication links, whereas \cite{guo2020actuator} places actuators to reduce the approximate control
energy metric discussed in \cite{EffortBounds}. However, both studies assume strongly connected network graphs. In this case the corresponding problem can be formulated as a matroid optimization, allowing the use of efficient greedy heuristics. Hence, the first goal of this paper is to extend the non-modular metric minimization problem under structural controllability constraints to arbitrary graphs.

In \cite{clark2012leader,guo2020actuator}, Forward Greedy algorithm (\textbf{FG}) iteratively adds the most beneficial node to the leader or the actuator position sets under the constraints of structural controllability and a cardinality upper bound. This algorithm provides a solution that approximately minimizes the given metrics. Although there exist performance guarantees for the metric value under the actuator position set returned by \textbf{FG} \cite{bian2017guarantees,guo2020actuator}, this algorithm does not always generate satisfactory results \cite{tihanyi2021multirobot,karaca2021performance}. To overcome this problem, several variants have been investigated to improve the performance, including the Continuous Greedy Algorithm \cite{sviridenko2014optimal} and the Randomized Greedy Algorithm \cite{gao2018randomized}. These methods can potentially enjoy better cost upperbounds than \textbf{FG} in a probabilistic manner, however, they do not offer any ex-post performance improvement. Thus, the second goal of this paper is to derive a heuristic algorithm which is guaranteed to perform at least as good as \textbf{FG} (if not, better), while maintaining its polynomial complexity.  

Another issue in actuator placement is related to susceptibility of the actuators to faults. In extreme conditions, the actuators may go offline and thus may not be able to achieve the expected performance. Related research is focused on robust optimization for the worst-case scenario \cite{tzoumas2018resilient,hou2021robust}, efficient methods for contingency analysis \cite{chu2021short}, and security indices for protection of vulnerable actuators \cite{milovsevic2020actuator}. To the best of our knowledge, no previous work considers maintaining structural observability/controllability in case of offline sensor/actuators. Hence, in this paper, the third goal is to study methods to deploy backup actuators to enhance robustness of the network for structural controllability.

Targeting the aforementioned goals, our contributions are as follow. First, in case of a network graph which is not necessarily strongly connected, we propose a method to select actuators and, under certain conditions, we guarantee the structural controllability of the resulting system. Second, based on \textbf{FG}, we propose a novel method, called Long-Horizon Forward Greedy Algorithm (\textbf{LHFG}). {The performance of the actuator position set derived is guaranteed to be at least as good as (if not, better than) that of \textbf{FG}. Third, we
formulate the problem of minimal backup placements. The goal is to find the minimal backup actuator set such that the system can maintain its structural controllability even when any of the selected actuators goes offline. We show that this problem is equivalent to the hitting set problem.}

The remainder of this paper is organized as follows. In Section \ref{sec: Preliminaries}, we introduce the problem formulation and preliminaries. In Section \ref{sec: extensions}, we provide an extension of actuator placement to arbitrary graphs and we present \textbf{LHFG}, which is shown to provide an improvement over \textbf{FG}. In Section \ref{sec: Robust actuator placement}, we propose an efficient scheme for finding positions to place backups. To verify our results, we present a numerical case study in Section \ref{sec: a numerical case study}. Some conclusions are provided in Section~\ref{sec: conclusions}.

\vspace{.1cm}
\textbf{\textit{Remarks on notation:}} For notational simplicity, we use $v$ and $\{v\}$ interchangeably for singleton sets. Within pseudocodes of algorithms, we introduce several symbols for sets and variables to describe the steps of the algorithms. These notations, without being defined in the main body of this paper, will be used in the proofs.  
\vspace{.1cm}

\section{Problem formulations and preliminaries}
\label{sec: Preliminaries}
\subsection{Problem formulations}
\label{sec: Problem formulation}
Consider a linear system with state vector $x\in \R^n$. To each state variable $x_i\in \R$, we associate a node $v_i\in V:= \{v_1, \ldots, v_n\}$. If we place actuators on the set of nodes $S\subset V$, called the actuator position set, and apply an input vector $u\in \mathbb{R}^n$, the system dynamics can be written as
\begin{equation}
\label{eq: systemmodel}
\dot{x} = Ax + B(S)u,
\end{equation}
where $B(S)= \text{diag}(\bm {1}(S))\in \R^{n\times n}$ and $\bm {1}(S)$ denotes a vector of size $n$ whose $i$th entry is $1$ if $v_i$ belongs to $S$ and $0$ otherwise. Let $G=(V,E)$ denote a weighted directed graph associated with the adjacency matrix $A$ with nodes $V$ and edges $E$, where $|E|=l$ and the directed edge $(v_j,v_i)\in E$ if the associated weight $(A)_{ij}$ is non-zero. 

The pair $(A,B(S))$ is called \textit{controllable} if
for all $x_0,x_1\in\R^n$ and $T>0$ there exists a control input $u:[0,\, T]\rightarrow\R^n$ that steers the system from $x_0$ at $t=0$ to $x_1$ at $t=T$. For linear time-invariant systems, controllability can be verified by the rank of the controllability matrix $P=\begin{bmatrix}
   B(S) & AB(S) & \cdots & A^{n-1}B(S)
\end{bmatrix}\in \R^{n\times n^2}$. Due to potential errors in the identification of the edge weights, most of the time we can only rely on the topology but not on the particular weights. Motivated by this particularity, we introduce the weaker notion of structural controllability.
\begin{definition}
\label{def: structuralcontrollability}
We say that $(A,B)$ and $(\hat{A},\hat{B})$ with $A,B,\hat{A},\hat{B}$ $\in \R^{n\times n}$ have the same structure if matrices $[A\text{ }B]$ and $[\hat{A} \text{ }\hat{B}]$ have zeros at the same entries. Given $S\subset V$, $(A,B(S))$ is \textit{structurally controllable} if there exists a controllable pair $(\hat{A},\hat{B})$ having the same structure as $(A,B(S))$.
\end{definition}

As shown in~\cite{SturcturalcontrollabilityTai}, structural controllability of the pair $(A,B)$  further implies that even when $(A,B)$ is not controllable, it is always possible to slightly perturb some non-zero edge weights to ensure controllability. 

With this notion, the first problem we aim to solve is the following.

\vspace{.1cm}
\textbf{P1}: Suppose a
potentially non-modular and non-increasing metric $f:2^V\rightarrow\R$ is given. Find $K$ actuators attaining structural controllability while minimizing the metric at hand:
\begin{equation}
\begin{aligned}
& \min_{S\subset V,|S|= K}
& & f(S)
\\ 
&\quad\ \ \mathrm{s.t.} & & (A,B(S)) \text{ is structurally controllable.} 
\end{aligned}
\label{eq:mainproblem1}
\end{equation}
\vspace{.1cm}

Suppose ${S}$ is a feasible solution to \textbf{P1}, which could be either optimal or approximately optimal. The second problem considers the robustness of $(A,B({S}))$ against failures, which will be further motivated in Section~\ref{sec: Robust actuator placement}.

\vspace{.1cm}
\textbf{P2}: Given ${S}$, find the minimal backup actuator positions such that structural controllability can be retained when any single actuator at $v\in {S}$ malfunctions:
\begin{equation}
\begin{split}\label{eq:mainproblem_2}
\min_{\mathcal{B}\subset V} & \quad \rvert \mathcal{B}\rvert \\
 \mathrm{s.t.} & \quad \forall v\in {S}, \exists b_v\in \mathcal{B}, (A,B(\{{S}\setminus v\} \cup b_v)) \text{ is }
\\
 & \quad \text{structurally controllable}.
\end{split}
\end{equation}
\vspace{.1cm}

To define the constraints on these two problems, we need an efficient characterization of structural controllability.
\subsection{Characterization of structural controllability}
\label{sec: Structural controllability}
Structural controllabilty boils down to checking two graphical properties. A system $(A,B(S))$ is structurally controllable if and only if it satisfies \textit{accessibility} 
and \textit{dilation-freeness}\cite[Theorem~1]{SturcturalcontrollabilityTai}, which are defined as follow.

\begin{definition}
The system $(A,B(S))$ satisfies accessibility (or the nodes in $V$ are accessible by $S$) if for any $v\in V$ there exists a path from a node in $S$ to $v$ in $G = (V,E)$.
\end{definition}

Note that one can use Breadth First Search to verify accessibility condition. For dilation-freeness, on the other hand, we need a method to distinguish a node with an actuator from other nodes in the graph $G = (V,E)$. Let
$S''=\{v''_{s_1},,\ldots,v''_{s_K}\}$ denote a copy of the actuator position set $S=\{v_{s_1},\ldots,v_{s_K}\}$ and $E''$ denote edges connecting $v''_{s_i}$ to $v_{s_i}$, $i=1,\ldots,K$. Then, the graph $G'' = (V\cup S'',E\cup E'')$ illustrates explicitly how the actuators are connected to the nodes in $V$.

\begin{definition}The system $(A,B(S))$ is dilation-free if in $G''$ the in-neighbor set $N(U):=\{\tilde{v}\,|\,\exists v_u\in V\cup S'' \text{ s.t. } (\tilde{v},v_u)\in E\cup E''\}$ for any subset $U\subset V$ satisfies $|N(U)|\geq |U|$.
\end{definition} 

Dilation-freeness requirement for controllability can be interpreted as follows. 
Whenever there are fewer actuators connected to a set of nodes than the cardinality of the set under consideration,
we do not possess the flexibility to steer the states of these nodes arbitrarily to achieve any controllability notion.

In the following, we introduce methods for dilation-freeness checks via matchings in bipartite graphs \cite{cormen2009introduction,plummer1986matching}. 

We define concepts related to bipartite graphs. An undirected graph is called bipartite and denoted as $(V^1, V^2, \mathsf{E})$ if its vertices are partitioned into $V^1$ and $V^2$, while any undirected edge in $\mathsf{E}$ connects a vertex in $V^1$ to another in $V^2$. A matching $m$ is a subset of $\mathsf{E}$ where no two edges in $m$ share a vertex in common. Given a subset $L\subset$ $V^1\cup V^2$, we say $L$ is covered by $m$ if any $v\in L$ is incident to an edge in $m$. The matching $m$ is called maximum if it has the largest cardinality among all possible matchings and is called perfect if $V^2$ is covered.

We utilize an auxiliary bipartite graph to check dilation-freeness. It is constructed as follows. Let node sets $V'=\{v'_1,\ldots, v'_n\}$ and ${V''}=\{v''_1,\ldots, v''_n\}$ be two copies of $V =\{{v}_1,\ldots, {v}_n\}$. As for the edges, the set $\mathsf{E}$ consists of undirected edges connecting $v_{i}$ with $v'_{j}$ if $(v_i,v_j)\in E$, whereas the edge set $\mathsf{E}_S$ consists of undirected edges connecting $v'_k$ with $v''_k$ if $v_k\in S$. The auxiliary bipartite graph is then given by $\mathcal{H}_b(S) := (V\cup S'', V', \mathsf{E}\cup \mathsf{E}_S)$. With this graph at hand, the following lemma \cite[Proposition~7]{guo2020actuator} provides an efficient method to check dilation-freeness.
\begin{lemma}[Dilation-freeness check]
Let $\mathcal{C}_K=\{S\subset V\,|\,|S| = K$ and the pair $(A,B(S))$ is dilation-free$\}$ and $\mathcal{\tilde{C}}_K=\{\Omega\,|\,\exists S \in \mathcal{C}_K \text{ such that } \Omega \subset S \}$. With $\bar{m}(S)$ as a maximum matching in $\mathcal{H}_b(S)$, we have
\begin{enumerate}[label=(\roman*)]
    \item $S\in \mathcal{C}_K$ if and only if $|S|=K$ and $\bar{m}(S) = n$,
    \item $S\in \mathcal{\tilde{C}}_K$ if and only if $|S|\leq K$ and $\bar{m}(S) \geq n-K+|S|$. 
\end{enumerate}
\end{lemma}

{Note that a maximum matching can easily be derived via Edmonds-Karp Algorithm with the complexity of ${O}(nl^2)$ \cite{edmonds1972theoretical}. We briefly discuss the motivation for defining $ \mathcal{\tilde{C}}_K$, a family of sets, and checking membership to it.
The set $\mathcal{\tilde{C}}_K$ contains all the sets that can be expanded to a set $S$ such that $|S|=K$ and $(A,B(S))$ is dilation-free. Checking membership with respect to $\mathcal{\tilde{C}}_K$ is essential to implementing the greedy algorithm to approximately solve \textbf{P1}, a popular polynomial-time heuristic in the literature to iteratively solve combinatorial optimization problems \cite{nemhauser1978analysis}. This method will be discussed in detail in Section \ref{sec: FG}.}

We illustrate these dilation-freeness checks with the following example.

\begin{example}
\label{exp: forwardFeas}
Consider a system described by $4$ nodes and the dynamic equations \eqref{eq: systemmodel} where
$${\small A = \begin{bmatrix}
   0 & -0.5 & -0.8 & -0.6\\
   1 & 0 & 0 & 0\\
   1 & 0 & 0 & 0\\
   1 & 0 & 0 & 0
\end{bmatrix}.}$$
The graph $G=(V,E)$ corresponding to the adjacency matrix $A$ is provided in Figure \ref{fig:system}. From the bipartite $\mathcal{H}_b(\emptyset)$ in Figure \ref{fig:bipartite}, we see that any maximum matching consists of 2 edges. Only if $K\geq 2$ do we have that $\bar{m}(\emptyset)\geq n-K+|\emptyset|$ and $\emptyset \in \tilde{\mathcal{C}}_K$, which is to say $\tilde{\mathcal{C}}_K \neq \emptyset$. Thus, $K=2$ is the minimum number of actuators required for dilation-freeness. Then we check the actuator set $S = \{v_3,v_4\}$. The bipartite graph $\mathcal{H}_b(S)$ is shown in Figure \ref{fig:bipartite}, from which we see the maximum matching contains $4$ edges and thus $S\in {\mathcal{C}}_2$.
\begin{figure}[t]
    \centering
    \includegraphics[width=0.3\linewidth]{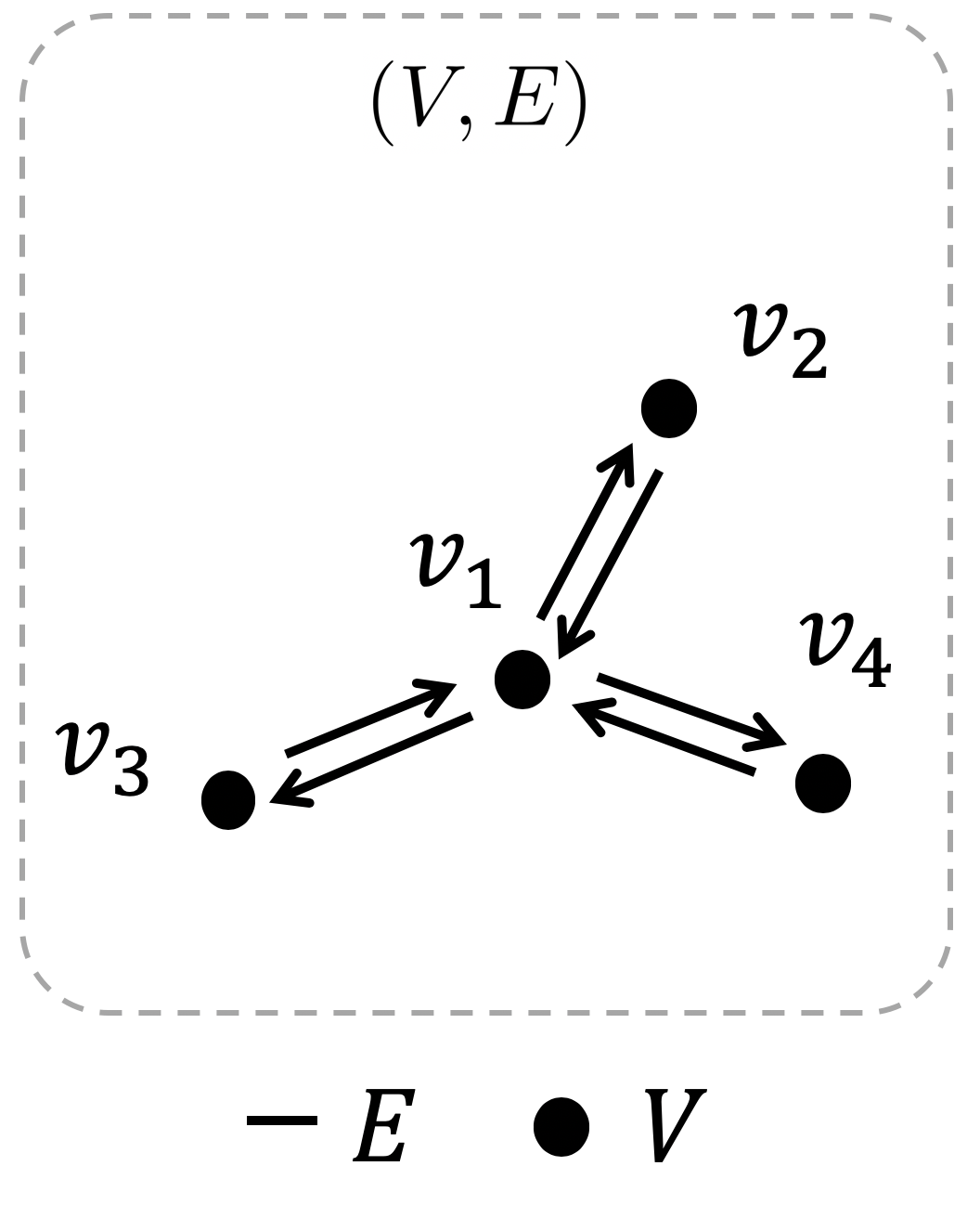}
    \caption{Graph for the 4-node system}
    \label{fig:system}
\end{figure}
\begin{figure}[t]
    \centering
    \includegraphics[width=0.6\linewidth]{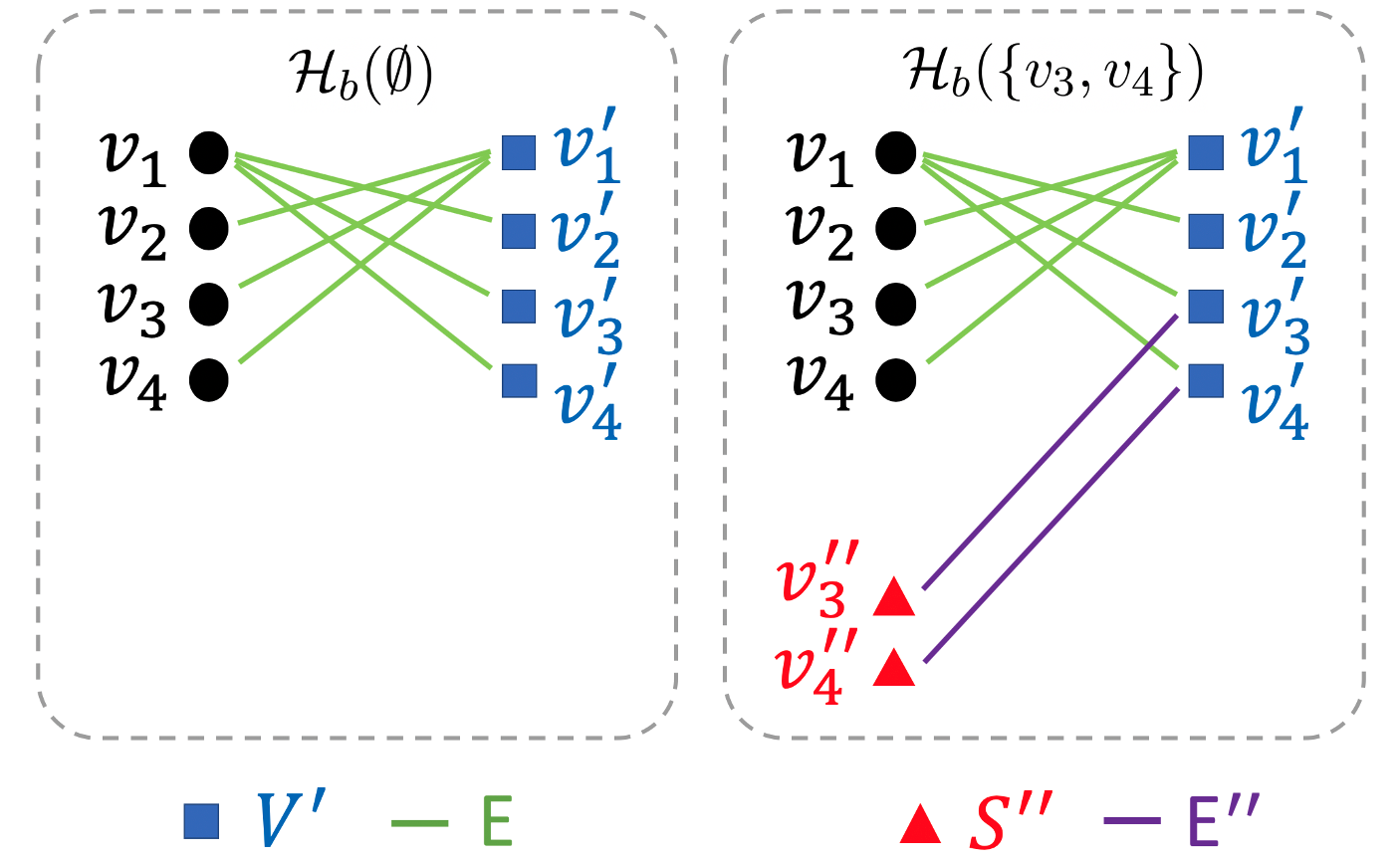}
    \caption{Auxiliary bipartite graphs $\mathcal{H}_b(\emptyset)$ and $\mathcal{H}_b(\{v_3,v_4\})$}
    \label{fig:bipartite}
\end{figure}
\end{example}

\subsection{Forward greedy algorithm for structural controllability}
\label{sec: FG}
A special instance of \textbf{P1} has been studied in \cite{clark2012leader,guo2020actuator}, where strong connectivity of $G$, the directed graph corresponding to the adjacency matrix $A$, is assumed. With this assumption, $(A,B(S))$ satisfies accessibility whenever $|S|>0$. Consequently, if $|S|\neq 0$, dilation-freeness implies structural controllability. \textbf{P1} is then reduced to minimization of $f(S)$ with the constraint $S\in \tilde{\mathcal{C}}_K$, which is a matroid optimization problem \cite{guo2020actuator}. 

The study in \cite{guo2020actuator} proposes to apply the Forward Greedy Algorithm (\textbf{FG}) through the function $\textproc{FG}(S^0,d)$ shown in Algorithm~\ref{alg:ALG_fg}. Among the inputs, $S^0$ denotes the initial set before expansion and $d$, called the \textit{depth}, denotes an upper bound for the number of expansions. In the pseudocode of Algorithm \ref{alg:ALG_fg}, the function $\textproc{IsMember}$ is implemented using the efficient graph-theoretical techniques introduced in Section \ref{sec: Structural controllability}. This algorithm can be justified by the following.

\begin{fact}[\cite{guo2020actuator}~]
If $G$ is strongly connected, with $d=K-|S_0|$ and $S^0\in\tilde{\mathcal{C}}_K$, $S_\mathsf{f} =\textproc{FG}(S^0,d)$ satisfies that $S_\mathsf{f}\in\mathcal{C}_K$.
\end{fact}

%The fact above follows from the observation that matroid satisfies the properties of independence systems~\cite{il2003hereditary}.
In addition to the fact above, one can utilize special properties of matroids~\cite{edmonds1971matroids} and provide a suboptimality bound on $f(S_\mathsf{f} =\textproc{FG}(\emptyset,K))$ with respect to $f(S^*)$, given the submodularity ratio and the curvature of the metric $f$~\cite{karaca2021performance}.
{However, these results rely on the strong connectivity of $G$. If the digraph $G$ is not strongly connected, dilation-freeness of $(A,B(S))$ does not necessarily lead to structural controllability. In view of this, the greedy solution may fail to attain structural controllability.}

In Section~\ref{sec: extensions} below, targeting \textbf{P1}, we 1) extend \textbf{FG} beyond the assumption of strong connectivity by providing an efficient method to attain the accessibility condition and 2) propose a novel algorithm based on \textbf{FG} with a longer horizon and ensuring a better performance. In Section 
\ref{sec: Robust actuator placement}, we address the solution to \textbf{P2}.

\begin{algorithm}[t]
        \caption{Forward Greedy Algorithm}
        \label{alg:ALG_fg}
        {\begin{algorithmic}
        \Require initial set $S^0$ and depth $d$ 
        \Ensure full actuator position set $S_\mathsf{f}$
        \Function {FG}{$S^0,d$}
        \State {$U^0= V\setminus S^0$, $t=1$}
        \While{$U^{t-1}\neq V$ {, } $|S^{t-1}|<K$ and $t\leq d$}
        \State{${v^*}= \arg\max_{i\in V\setminus U^{t-1}}f(S^{t-1})-f(S^{t-1}\cup i)$}
            \If {\textproc{IsMember}($S^{t-1}\cup v^*, \tilde{\mathcal{C}}_K$)}
        \State $v_\mathsf{f}^t \gets v^*$
        \State $S^{t} \gets S^{t-1}\cup v_\mathsf{f}^t $ and $U^{t}\gets U^{t-1}\cup v_\mathsf{f}^t  $
        \State $t\gets t+1$
        \Else
        \State $U^{t-1}\gets U^{t-1}\cup v^* $
        
    \EndIf    
            \EndWhile
            \State $S_\mathsf{f} \gets S^{t-1}$
        \EndFunction
        \end{algorithmic}}
    \end{algorithm}

\section{Extensions to the greedy heuristics}
\label{sec: extensions}
\subsection{Beyond strong connectivity}
In case the graph $G=(V,E)$ is not strongly connected, we consider finding an initial set $S_0$ before running \textbf{FG} such that all the nodes in $V$ are accessible by $S_0$. 

We start by highlighting that the well-establish Kosaraju's algorithm \cite{cormen2009introduction} can construct all the strongly connected components in the directed graph $G$ in linear time, that is, with complexity $O(n+l)$, where $l$ is the number of edges in $G$. Whenever $G$ is not strongly connected, there would be components with no incoming edges and we denote them as $\mathcal{T}^1,\ldots,\mathcal{T}^{n_\mathcal{T}}$. If all the nodes in $V$ are accessible by a set $S$, then for any $1\leq i\leq n_\mathcal{T}$ the component $\mathcal{T}^i$ guarantees that $\mathcal{T}^i\cap S\neq \emptyset$, otherwise the nodes in $\mathcal{T}^i$ would not be accessible. Based on this idea, we propose Algorithm \ref{alg:wms} to construct $S^0$ when $G$ is not strongly connected.

The following proposition proves that under certain conditions one can use Algorithm \ref{alg:ALG_fg} to further expand $S^0$ such that the derived system is structurally controllable.

\begin{algorithm}[t]
        \caption{Initial Set Construction}
        \label{alg:wms}
        {\begin{algorithmic}
        \Require graph $G=(V,E)$
        \Ensure an actuator position set $S^0$
        \Function {IniS}{$G$}
        \State {Find all strongly connected components with no incoming edges, denoted as $\mathcal{T}^i=\{v^i_1\ldots,v^i_{p_i}\}$} for any $1\leq i \leq n_\mathcal{T}$ using Kosaraju’s algorithm
        \State {$S^{0,0} = \emptyset$}
        \For{\texttt{$t = 1,\ldots,n_\mathcal{T}$}}
            \State Find $\mathcal{P}^t = \{v\in \mathcal{T}^t|S^{0,t-1}\cup v\in \tilde{\mathcal{C}}_K\}.$
            \State ${v^t_{j^t_*}}= \arg\min_{v^t_{j}\in \mathcal{P}^t}f(S^{0,t-1}\cup v^t_{j})$
            \State $S^{0,t} \gets S^{0,t-1}\cup v^t_{j^t_*}$
            \EndFor
            \State $S^{0} \gets S^{0,n_\mathcal{T}}$
        \EndFunction
        \end{algorithmic}}
    \end{algorithm}
    
\begin{proposition}
\label{prop: arbitrary_graphs }
Let $k$ be the smallest integer such that $\mathcal{C}_k\neq \emptyset$. If $n_\mathcal{T}$ satisfies $n> K\geq k + n_\mathcal{T}$, then\begin{enumerate}[label=(\roman*)]
    \item the set $S^0$ derived through Algorithm 2 belongs to $\mathcal{\tilde{C}}_K$,
    \item By applying Algorithm~\ref{alg:ALG_fg} to further expand $S^0$ and obtaining $S_\mathsf{f} = \textproc{FG}(S^0,K-|S^0|)$, the resulting system $(A,B(S_\mathsf{f}))$ is structurally controllable.
    \end{enumerate}
\end{proposition}
\begin{proof}
For the proof of (i), for any $t$, we verify that $v^t_{j^t_*}$ is well defined and we then show that $S^{0,t}\in \tilde{\mathcal{C}}_K$. 

Let $W$ be a set in $ \mathcal{C}_k$ and notice that for any $S\subset V$ the system $(A,B(W\cup S))$ is dilation-free. In Algorithm~\ref{alg:wms}, since for $1\leq t \leq n_\mathcal{T}$ and any $v\in V$ we have $|S^{0,t-1}\cup v|\leq  n_\mathcal{T}$, $S^{0,t-1}\cup v\cup W\in \mathcal{\tilde{C}}_K$ and thus $\mathcal{P}^t\neq \emptyset$, which means for each $t$ we can derive $v^t_{j^t_*}\in \mathcal{P}^t$. Suppose $S^{0,t-1}\in \tilde{\mathcal{C}}_K$, then $S^{0,t} = S^{0,t-1}\cup v^t_{j^t_*}\in \mathcal{\tilde{C}}_K$. By induction, $S^0\in\mathcal{\tilde{C}}_K$.

Regarding the statement (ii), to prove the structural controllability of $(A,B(S_\mathsf{f}))$, we need to show this system satisfies both accessibility and dilation-freeness. 

Since for any $t$ $v^t_{j^t_*}\in \mathcal{T}^t$, $S^0\cap \mathcal{T}^t \neq \emptyset$. From graph theory, any node $v\in V$ is accessible by $S^0\subset S_\mathsf{f}$. The method to prove dilation-freeness follows closely that of ``$S^{0,t}\in \mathcal{\tilde{C}}_K$" provided above, and thus omitted.\hfill\QEDA
\end{proof}

In case a large number of actuators is allowed, the condition ``$n> K\geq k + n_\mathcal{T}$'' is not restrictive. Whenever this condition does not hold, we may still obtain $S^0\in \tilde{\mathcal{C}}_K$, and we can then use \textbf{FG} to expand $S^0$.

To summarize, in case $G$ is not strongly connected, our overall method for actuator placement is to first run Algorithm~\ref{alg:wms} to obtain the initial set $S^0=\textproc{IniS}(G)$ and then to execute Algorithm~\ref{alg:ALG_fg} to obtain the actuator position set $S_\mathsf{f} = \textproc{FG}(S^0,K-|S^0|)$. Since this procedure is not the standard greedy algorithm, the upper bounds for the suboptimality gaps in previous works are not applicable. {We leave deriving a guarantee towards the performance of the actuator position set $\textproc{FG}(\textproc{IniS}(G),K-|S_0|)$ as a future work.}

\subsection{Beyond myopic decisions}
\label{sec: Long-horizon Greedy Algorithm}
This section derives a method that improves upon \textbf{FG} in terms of the performance of the final actuator position set.

For the iterations in Algorithm \ref{alg:ALG_fg}, the marginally most beneficial nodes, $v^1_\mathsf{f},\ldots, v^{K-|S^0|}_\mathsf{f}$, are added one after the other. These decisions are myopic, since a node added at some iteration might make it harder to further decrease the given metric in later iterations. To mitigate this issue, we equip \textbf{FG} with a long horizon, as shown in Algorithm \ref{alg:ALG_lfg}. This way, we can evaluate whether to add a node based on how the addition of this node influences the actuator position set expansion in the future. 

We call this new method Long-Horizon Greedy Algorithm (\textbf{LHFG}). We achieve the longer horizon by using \textbf{FG} to further expand starting from the node we are considering. The depth of this embedded \textbf{FG} is called the horizon of \textbf{LHFG}.

\begin{algorithm}[t]
        \caption{Long-horizon Greedy Algorithm}
        \label{alg:ALG_lfg}
        {\begin{algorithmic}
        \Require actuator position set $S^0$ and horizon $d_\mathrm{h}$
        \Ensure an actuator position set $S_{\mathsf{lf}}$
        \Function {LHFG}{$S^0,d_\mathrm{h}$}
        \State {$t\gets1$} and $S_{\mathsf{lf}}^0 = S^0$
        \While{ $|S_{\mathsf{lf}}^{t-1}|<K$} 
            \State Find $\mathcal{R}^t = \{v $ $| \textproc{ IsMember}(S_{\mathsf{lf}}^{t-1}\cup v,\tilde{\mathcal{C}}_K) = 1\}$
            \State Calculate by enumeration
            {${v_{\mathsf{lf}}^t}= \arg\min_{v\in \mathcal{R}^t}f(\textproc{FG}(S_{\mathsf{lf}}^{t-1}\cup v,d_\mathrm{h}))$}
            \State $S_{\mathsf{lf}}^{t} \gets S_{\mathsf{lf}}^{t-1}\cup v_{\mathsf{lf}}^t $ and $t\gets t+1$
            \EndWhile
            \State $S_\mathsf{lf} \gets S_\mathsf{lf}^{t-1}$
        \EndFunction
        \end{algorithmic}}
\end{algorithm}

The following proposition proves that with a horizon long enough, \textbf{LHFG} has a performance no worse than \textbf{FG}.
\begin{proposition}
For an arbitrary network, if $S^0\in \tilde{\mathcal{C}}_K$, $f(S_{\mathsf{lf}})\leq f(S_{\mathsf{f}})$, where $S_{\mathsf{lf}} = \textproc{LHFG}(S^0,K-|S^0|)$ and $S_{\mathsf{f}} = \textproc{FG}(S^0,K-|S^0|)$.
\end{proposition}
\begin{proof}
Following the proof of Proposition \ref{prop: arbitrary_graphs } regarding dilation-freeness, one can see that $S_\mathsf{lf}$ and $S_\mathsf{f}$ are both well-defined and belong to the set $\mathcal{C}_K$. Next, we show $S_\mathsf{lf}$ achieves a lower metric value. 

By denoting $\bar{S}^{t}_{\mathsf{lf}}:=\textproc{FG}(S_{\mathsf{lf}}^{t},K-|S^0|)$, we claim that $f(\bar{S}^{t}_{\mathsf{lf}}) \geq f(\bar{S}^{t+1}_{\mathsf{lf}}) $ for $t<K-|S_0|$. To prove this, we notice that $\bar{S}^{t}_{\mathsf{lf}}\setminus S_{\mathsf{lf}}^{t} \subset \mathcal{R}^{t+1}$ and there exists $\tilde{v}\in \bar{S}^{t}_{\mathsf{lf}}\setminus S_{\mathsf{lf}}^{t}$ such that $\tilde{v}\in \mathcal{R}^{t+1}$ and $\bar{S}^{t}_{\mathsf{lf}} = \textproc{FG}(S_{\mathsf{lf}}^{t}\cup \tilde{v},K-|S^0|)$, therefore, $f(\bar{S}^{t+1}_{\mathsf{lf}}) = \arg\min_{v\in \mathcal{R}^{t+1}}f(\textproc{FG}(S_{\mathsf{lf}}^{t}\cup v,K-|S^0|))\leq f(\textproc{FG}(S_{\mathsf{lf}}^{t}\cup \tilde{v},K-|S^0|)) = f(\bar{S}^{t}_{\mathsf{lf}})$. 

To proceed, we recall that $v_{\mathsf{lf}}^1,\ldots,v_{\mathsf{lf}}^{K-|S_0|}$, in order, are added to form $S_{\mathsf{lf}}$ while $v_{\mathsf{f}}^1,\ldots,v_{\mathsf{f}}^{K-|S_0|}$, in order, are added to form $S_{\mathsf{f}}$. If $i$ is the smallest integer such that $v_{\mathsf{lf}}^i \neq v_{\mathsf{f}}^i$, $f(S_{\mathsf{f}}) = f(\bar{S}^{i-1}_{\mathsf{lf}})\geq f(\textproc{FG}(S_{\mathsf{lf}}^{i-1}\cup v^i_{\mathsf{lf}},K-|S^0|))= f(\bar{S}^{i}_{\mathsf{lf}})\geq \bar{S}^{K-|S^0|}_{\mathsf{lf}}$. By noticing that $S_{\mathsf{lf}} = \bar{S}^{K-|S^0|}_{\mathsf{lf}}$, we have $f(S_{\mathsf{lf}})\leq  f(S_{\mathsf{f}})$. \hfill\QEDA
\end{proof}

We cannot guarantee that $f(S_{\mathsf{lf}})< f(S_{\mathsf{f}})$ because of two potential cases, i) $S_{\mathsf{lf}} = S_{\mathsf{f}}$ and ii) $f(S_{\mathsf{lf}})= f(S_{\mathsf{f}})$ even if $S_{\mathsf{lf}} \neq S_{\mathsf{f}}$. However, in practice, we expect these two cases to be rare. We refer the readers to Section \ref{sec: a numerical case study} for a numerical case study where $f(S_{\mathsf{lf}})$ is significantly smaller than $f(S_{\mathsf{f}})$.

%Since we have no assumptions on the complexity to compute $f(S)$, we cannot conduct a rigorous analysis of the computational complexity. 
Notice that, to derive $S_{\mathsf{lf}} = \textproc{LHFG}(\textproc{IniS}(G),K-|S_0|)$, most of the computational time would be spent on long horizon evaluations. During these evaluations, we execute \textbf{FG} for at most  $K(2n-K+1)/2$ times. Suppose the complexity of calculating $f(S)$ is ${O}(q(n))$. Note that the complexity of $\textproc{ISMEMBER}$ is ${O}(nl^2)$ from Edmonds-Karp algorithm. One can then verify that the complexity of \textbf{LHFG} is ${O}(Kn^3(q(n)+l^2))$. To reduce the computational complexity, one can shorten the horizon $d_\mathrm{h}$ of \textbf{LHFG}. However, it is then not possible to guarantee better performance than \textbf{FG}, as in the proposition above. In Section \ref{sec: a numerical case study}, we study the computational time and the derived actuator sets in a numerical example for $d_\mathrm{h}<K-|S_0|$.

\section{Backup placements for ensuring structural controllability in response to failures}
\label{sec: Robust actuator placement}

In many applications, the selected actuators $S$, that we call \textit{primary}, derived through \textbf{LHFG} may be offline due to potential damages/failure. An offline primary actuator can potentially make the system uncontrollable, which would be unacceptable. In this section, we first list two assumptions that motivate \textbf{P2}. We then show that \textbf{P2} is equivalent to hitting set problem, which has been well studied in the combinatorial optimization literature. 

\begin{assumption}
\label{ass: one actuator failure}
Only one actuator at a time can be offline.
\end{assumption}

This assumption can hold if primary actuator failures are not frequent and/or offline actuators can be restored quickly. If several actuators can go offline at the same time, at the current stage of our research, we need to enumerate all possible combinations of failures and it would be quite conservative to deploy backups for the worst case scenario.

\begin{assumption}
\label{sec: necessity}
There exists at least one actuator $v$ such that $(A,B(S\setminus v))$ is not structurally controllable.
\end{assumption}

A primary actuator is called \textit{essential} if its failure violates the structural controllability. In practice, even if the number of primary actuators is more than the minimum needed for structural controllability, it is often the case that the assumption above holds. We illustrate this phenomenon in Section \ref{sec: a numerical case study} with a numerical example.

Due to Assumption \ref{sec: necessity}, it is necessary to have backups for essential primary actuators. Under these assumption, \textbf{P2} in Section \ref{sec: Problem formulation} provides us with the minimal backup position set $\mathcal{B}$. With these {backup} actuators of $\mathcal{B}$, we can still retain structural controllability by replacing any single offline primary actuator.

To solve \textbf{P2}, we need to characterize its constraint set in a tractable manner. For this purpose, following definitions are in order.

\begin{definition}
Given the actuator position set $S$ such that $(A,B(S))$ is structurally controllable, we say $v$ is a DFR (dilation-freeness-recovering) backup position for $v_\mathrm{off}\in S$ if $(A,B(S\setminus v_\mathrm{off} \cup v))$ is dilation-free. We say $v$ is a feasible backup position for $v_\mathrm{off}$ if $(A,B(S\setminus v_\mathrm{off} \cup v))$ is structurally controllable.
\end{definition}

We will now provide a tractable characterization of the feasible backup positions for a given $v_\mathrm{off}$. To start with, consider the DFR backup positions. Recall that $S\in \mathcal{C}_K$ if and only if there exists a perfect matching $m_0$ in the bipartite graph $\mathcal{H}_b(S)$. With the primary actuator at $v_\mathrm{off}$ going offline, checking whether node $v$ is a DFR backup position is equivalent to checking whether there exists a perfect matching in $\mathcal{H}_b(S\setminus v_\mathrm{off}\cup v)$. For any $v\in V$, one can run the Edmonds-Karp Algorithm on $\mathcal{H}_b(S\setminus v_\mathsf{off}\cup v)$. By iteratively doing so, we can obtain all DFR backup positions for $v_\mathrm{off}$, with the complexity of $n^2l^2$. Such a naive approach fails to exploit the properties of matchings in bipartite graphs.

The following theorem characterizes all the DFR backup positions for $v_\mathrm{off}$ in a computationally efficient way.

\begin{theorem}
\label{thm: characterization_of_feasible_back_ups}
Let $S\in{\mathcal{C}}_K$ be the set of actuators. Suppose $v_\mathrm{off}\in S$ is offline and dilation freeness is lost, that is, there does not exist a perfect matching in $H_1 := \mathcal{H}_b(S\setminus v_\mathrm{off})$. A node $v\neq v_\mathrm{off}$ is DFR for $v_\text{off}$ if and only if there exist a perfect matching $m_0$ in $\mathcal{H}_b(S)$ and an alternating path $p = (v'_\mathrm{off},v_{p_1},v'_{p_2},\ldots,v_{p_{r}},v')$ where edges $(v'_\text{off},v_{p_1}),(v'_{p_2},v_{p_3}),\ldots,(v'_{p_{r-1}},v_{p_r}) \notin m_0 $, $(v_{p_1},v'_{p_2}),\ldots,(v_{p_r},v')\in m_0$.
\end{theorem}
\begin{proof}
For both parts, observe that since there does not exist a perfect matching in $\mathcal{H}_b(S\setminus v_\mathrm{off})$, the matching $m_0$ must contain the edge $(v''_{\mathrm{off}},v'_{\mathrm{off}})$.

\textbf{Sufficiency}:  By utilizing the alternating path $p$, we exclude from $m_0$ the edges $(v''_{\mathrm{off}},v'_{\mathrm{off}})$, $(v_{p_1},v'_{p_2}),\ldots,(v_{p_r},v')$ and include $(v'_\text{off},v_{p_1}),(v'_{p_2},v_{p_3}),\ldots,(v'_{p_{r-1}},v_{p_r}), (v',v'')$ to form a new edge set $m_2$. One can verify that all the edges in $m_2$ are contained in $H_2 := \mathcal{H}_b(S\setminus v_\mathrm{off}\cup v)$ and $m_2$ is a perfect matching in $H_2$, in other words, $v$ is a DFR backup for $v_\text{off}$.

\textbf{Necessity}: Since $v$ is a DFR backup position for $v_\mathrm{off}$, there exists a perfect matching in $H_2$. In this bipartite graph with the non-perfect matching $m_0\setminus (v''_\mathrm{off},v'_\mathrm{off})$, there exists an augmenting path $p_\alpha$ and by augmentation on this path one can obtain in $H_2$ a perfect matching \cite{cormen2009introduction}, denoted as $m_\alpha$. Due to the fact that there does not exist a perfect matching in $H_1$, the node $v''$ must be incident to $p_\alpha$. Moreover, the node $v'_\mathrm{off}$ is also incident to $p_\alpha$, otherwise the perfect matching $m_\alpha$ formed by augmentation cannot cover $v'_\mathrm{off}$ which contradicts the perfectness. We trim $p_\alpha$ to form $p_\beta$ such that $p_\beta$ only contains the part of $p_\alpha$ between $v'_\mathrm{off}$ and $v'$. One can verify that $p_\beta$ is an alternating path in $\mathcal{H}_b(S)$ with respect to the matching $m_0\setminus (v''_\mathrm{off},v'_\mathrm{off})$ and satisfies the characterization specified in the theorem.\hfill\QEDA
\end{proof}

\begin{remark}
If the assumptions of Theorem \ref{thm: characterization_of_feasible_back_ups} do not hold, i.e., there exists a perfect matching in $H_1$, the actuator $v_{\mathrm{off}}$ going offline does not affect the dilation-freeness property. Hence, this actuator is not essential and any node can be a DFR backup position.
\end{remark}

To find all the alternating paths described in Theorem \ref{thm: characterization_of_feasible_back_ups}, we can use Breadth-First Search, whose complexity is $O(n+l)$. This allows us to efficiently find the DFR backup position set $\mathcal{D}(v)$, for any $v\in V$. Note that $v\in \mathcal{D}(v)$. 

It is then tractable to derive the feasible backup position set $\mathcal{F}(v_{\mathrm{off}})$ from the DFR backup position sets. It holds that $\mathcal{F}(v_{\mathrm{off}}) = \mathcal{D}(v_{\mathrm{off}})$ except for the following situation. Recall that through Algorithm \ref{alg:wms} we derive the strongly connected components with no incoming edges $\mathcal{T}^1,\ldots,\mathcal{T}^{n_{\mathcal{T}}}$. If there exists $i\in \mathbb{Z}^+$ such that $v_{\mathrm{off}}\in \mathcal{T}^i$ and $S\cap \mathcal{T}^i = v_{\mathrm{off}}$, then the actuator at $v_{\mathrm{off}}$ is the only one in $\mathcal{T}^i$. In case it goes offline, the system $(A,B(S\setminus v_{\mathrm{off}}))$ no longer satisfies accessibility. For this specific $v_{\mathrm{off}}$, we should have $\mathcal{F}(v_{\mathrm{off}}) = \mathcal{D}(v_{\mathrm{off}}) \cap \mathcal{T}^i$. 

With the feasible backup position sets $\mathcal{F}(v_1),\ldots,\mathcal{F}(v_K)$ obtained  via Theorem~\ref{thm: characterization_of_feasible_back_ups}, \textbf{P2} can be reformulated as follows.

\vspace{.1cm}
\begin{corollary}
\label{coro: hitting set problem}
Given the primary actuator positions $S$, the optimization problem \textbf{P2} is equivalent to the following:
\begin{equation}
\label{eq:mainproblem_2_reformulation}
\begin{split}
\min_{\mathcal{B}\subset V} & \quad \rvert \mathcal{B}\rvert \\
 \mathrm{s.t.} & \quad \forall v\in S, \exists b_v\in \mathcal{B}, \text{ with } b_{v}\in \mathcal{F}(v).
\end{split}
\end{equation}
\end{corollary}
\vspace{.1cm}

This is the classical hitting set problem \cite{angel2009minimum}, which is well-known to be NP-hard. There are extensive studies  proposing efficient approximate solutions with provable approximation ratios, e.g., the LP-based approach in \cite{krivelevich1997approximate} and the randomized algorithm in \cite{el2014randomised}. 

As a summary, our overall method is to first find the primary actuators using the methods in Section~\ref{sec: extensions}.
We then construct the feasible backup position sets for the essential ones. Finally, we solve the hitting set problem.
\vspace{.3cm}
\section{A numerical case study}
\label{sec: a numerical case study}
We test our algorithms on a linear system whose system matrix $A$ corresponds to the digraph $G$ illustrated in Figure \ref{fig:graph}.\footnote{\text{The code for this numerical case study is publicly available at}\newline {\href{https://github.com/odetojsmith/Actuator-Placement-beyond-SC-and-towards-Robustness}{\textcolor{blue}{https://github.com/odetojsmith/Actuator-Placement-beyond-SC-and-towards-Robustness}}}} In this graph, the edge weights are set to be all $1$ and there are $5$ strongly connected components. 

\begin{figure}[t]
    \centering
    \includegraphics[width = 10.42cm]{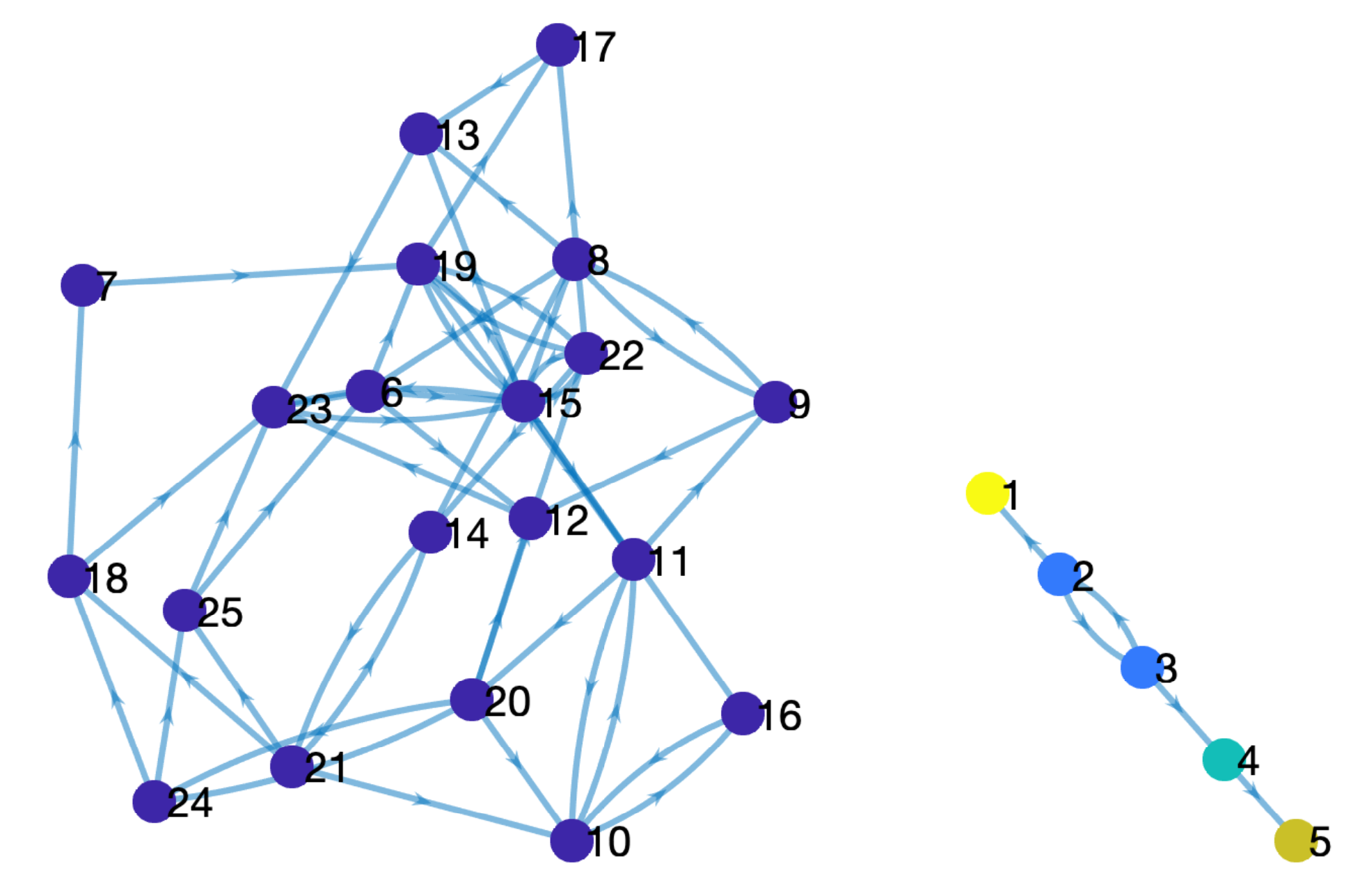}
    \caption{Digraph $G$ under consideration: strongly connected components are marked by different colors}
    \label{fig:graph}
\end{figure}

The metric under investigation is an approximate controllability metric $f = F_\epsilon$, where $F_{\epsilon}(S) =  \text{tr}((W_T(S)+\epsilon I)^{-1})$, and  $W_T(S)=\int_{0}^{T} e^{A\tau}B(S) B^\top (S) e^{A^\top\tau} d\tau$ is the controllability Grammian. This metric measures the average energy required for steering the system from $x_0$ with $||x_0||_2=1$ at $t=0$ to zero state at $t=T$. The constant term $\epsilon I$ allows for \textbf{FG} and \textbf{LHFG} to evaluate an uncontrollable actuator position set $S^0\notin \mathcal{C}_K$. We kindly refer the interested reader to \cite{tzoumas2018resilient,guo2020actuator} for more discussions on the metric $F_\epsilon(S)$. 
We let $\epsilon = 10^{-12}$ and $T=1$. Our goal is to find an actuator set $S$ with cardinality $9$ that minimizes the metric $F_\epsilon(S)$ while ensuring structural controllability of the resulting  system. 
\newpage

The initial actuator position set derived through Algorithm \ref{alg:wms} is $S^0 = \{16,2\}$, which is straightforward from the observation that, apart from the big strongly connected component containing nodes from 6 to 25, the node set $\{2,3\}$ is the only strongly connected component without any incoming edges. 

The actuator position sets derived through the forward greedy algorithm and the long-horizon greedy algoithm are respectively $S_\mathsf{f} = \textproc{FG}(S^0,\infty) = \{16,2,8,18,11,3,12,5,1\}$ and $S_\mathsf{lf} = \textproc{LHFG}(S^0,$
$\infty) = \{16, 2, 1, 13, 5, 8, 24, 14, 18\}$. The metrics under these two sets are $F_\epsilon(S_\mathsf{f}) = 7.56\times10^6$ and $F_\epsilon(S_\mathsf{lf}) = 1.08\times10^5$. Long horizon greedy provides $98.6\%$ improvement. The cost we pay for this perforance improvement is the computation time. It is given by 84.7 seconds for deriving $S_\mathsf{lf}$, whereas 1.3 seconds for deriving $S_\mathsf{f}$.\footnote{The case study is executed in Matlab R2019b on a computer equipped with 32GB RAM and a 2.3GHz Intel Core i9 processor.} To reduce the computational complexity, we let the horizon of \textbf{LHFG} to be $d_\mathrm{h}=3$ and obtain $S'_\mathsf{lf} = \textproc{LHFG}(S^0,3) = \{16, 2, 25, 1, 12, 5, 8, 20, 24\}$ with $F_\epsilon(S'_\mathsf{lf} ) = 1.34 \times 10^6$. As is discussed before, we do not have any performance guarantees for such modifications. The computation time for $S'_\mathsf{lf}$ is $64.6$ seconds. 

By checking the structural controllability when a single primary actuator gets offline, we find that the actuators at Node 1 and Node 2  are essential and they require backups. The feasible backup position sets are derived as $\mathcal{G}(v_1) = \{v_1,v_3\}$ and  $\mathcal{G}(v_2) = \{v_2\}$. Thus, for this system, one can select the backup position set as $\{v_2,v_3\}$. Suppose we place primary actuators at $S_\mathsf{lf}$ and $v_1$ goes offline. By activating the backup actuators at $v_3$, the metric is $1.07 \times 10 ^5$. In other words, this numerically verifies that the replacement do not jeopardize performance.
As a remark, in this example, the minimum number of actuators for structural controllability is $k=3$, less than $K=9$. Even then we see that there are essential actuators, which can result in the loss of structural controllability. Thus, it is necessary to detect such essential actuators and deploy backups.
\newpage

\section{Conclusion}
\label{sec: conclusions}
{In this paper, we studied the actuator placement problem minimizing a nonsubmodular and a nonsupermodular metric under the constraint of structural controllability. We extended the forward greedy algorithm (\textbf{FG}) to be applicable to arbitrary graphs and we proposed a novel algorithm, \textbf{LHFG}, which was proven to outperform \textbf{FG}. Then, to achieve robustness, we studied the minimal backup actuator placement problem and we showed that it is equivalent to the NP-hard hitting set problem.} 

{Our future work will focus on improving the computational complexity of \textbf{LHFG} and studying backup placement problem in case several primary actuators can go offline, simultaneously.}

\setstretch{1.05}
\bibliographystyle{IEEEtran}
\bibliography{IEEEabrv,library}
\vspace{.5cm}

\end{document}